\documentclass[12pt]{amsart}

\usepackage{enumerate}
\usepackage{amsfonts}
\usepackage{amsthm}
\usepackage{amsmath}
\usepackage{amssymb}
\usepackage{amsrefs}
\usepackage{t1enc}
\usepackage{fullpage}
\usepackage{dsfont}
\usepackage{graphicx}
\usepackage{mathrsfs}
\usepackage{mathptmx}
\usepackage[a4paper,left=2.5cm,right=2.5cm, footskip=30pt]{geometry}
\usepackage{lipsum}
\usepackage{geometry}

\newtheorem{theorem}{Theorem}[section]
\newtheorem{prop}[theorem]{Proposition}
\newtheorem{lemma}[theorem]{Lemma}

\newtheorem{kov}[theorem]{Corollary}
\theoremstyle{definition}
\newtheorem{megj}[theorem]{Remark}

\newcommand{\inte}{\mathop{\mathrm{int}}}

\begin{document}

\author{Bal\'azs Maga}

\address{E\"otv\"os Lor\'and University, P\'azm\'any P\'eter s\'et\'any 1/C, Budapest, H-1117 Hungary}
\email{magab@cs.elte.hu}

\title{Baire Categorical Aspects of First Passage Percolation}

\thanks{The author was supported by the \'UNKP-17-2 New National Excellence of the Hungarian Ministry of Human Capacities, and by the Hungarian National Research, Development and Innovation Office--NKFIH, Grant 124003.}


\begin{abstract}
In the previous decades, the theory of first passage percolation became a highly important area of probability theory. In this work, we will observe what can be said about the corresponding structure if we forget about the probability measure defined on the product space of edges and simply consider topology in the terms of residuality. We focus on interesting questions arising in the probabilistic setup that make sense in this setting, too. We will see that certain classical almost sure events, as the existence of geodesics have residual counterparts, while the notion of the limit shape or time constants gets as chaotic as possible.
\end{abstract}

\maketitle

\section{Introduction}

First passage percolation was introduced by Hammersley and Welsh in 1965 as a model to describe fluid flows through porous medium. It quickly became a popular area of probability theory, as one can easily ask very difficult questions. Many of these have still remained unsolved despite the growing interest from mathematicians, physicists and biologists. The main setup is the following: we have a given graph, usually we like to consider the lattice $\mathbb{Z}^d$. We denote the set of nearest neighbor edges by $E$. We place independent, identically distributed, non-negative random variables with a distribution law $\mu$ on each edge $e\in{E}$, which is called the passage time of $e$, and denoted by $\tau(e)$. We think about it as the time needed to traverse $e$. Based on this, we can define the passage time of any finite path $\Gamma$ of consecutive edges as the sum of the passage times of contained edges:
\begin{displaymath}
\tau(\Gamma)=\sum_{e\in\Gamma}\tau(e).
\end{displaymath}
\noindent Using this definition, we might define the passage time between any two points, or in other words the $T$-distance of any two points $x,y\in\mathbb{R}^d$
\begin{displaymath}
T(x,y)=\inf_{\Gamma}\tau(\Gamma),
\end{displaymath}
\noindent where the infimum is taken over all the paths connecting $x'$ to $y'$, where $x'$ and $y'$ are the unique lattice points such that $x\in x'+[0,1)^d$, $y\in y'+[0,1)^d$. The term "distance" is appropriate here: one can easily show that $T:\mathbb{Z}^d\times\mathbb{Z}^d\to\mathbb{R}$ is a pseudometric, that is an "almost metric" in which the distance of distinct points might be 0.

In brief, this is the probabilistic setup. In the sequel when we recall results related to this theory, for the sake of brevity we will often omit the precise technical conditions, such as conditions about the finiteness of certain moments or the value of the distribution function in the infimum of its support.  Instead of it we will simply refer to "some mild conditions" about the distribution function and cite the source of the result. For the reader interested in the details the recent survey paper \cite{FPP} is also warmly recommended.

By a similar virtue, we can define the topological setup: instead of non-negative random variables on each edge, we consider some $A\subseteq{\mathbb{R}_{\geq{0}}}$. To exclude trivialities, let $A$ have at least two elements. The passage time of any edge will be an element of $A$, and passage times of paths and between points are defined as in the probabilistic setup. Formally, the space of configurations is $\Omega=\times_{e\in{E}}A$. To define topology, we equip $A$ by its usual subspace topology inherited from $\mathbb{R}$, and equip $\Omega=\times_{e\in{E}}A$ with the product topology. If there might be ambiguity, we will write $T_\omega$ and $\tau_\omega$ for the passage times in the $\omega\in\Omega$ configuration. 

Now we are interested in the classical questions of the probabilistic setup which make sense in the topological setup as well. More precisely, we examine whether a property which has probability 1 in the probabilistic setup holds in a residual set of $\Omega$ in the topological setup. For example, as it was proved in \cite{WR} for $d=2$ and any distribution, and in \cite{K} for arbitrary $d$ under mild conditions on the distribution, with probability 1 there exists an optimal path between any two lattice points, which is called a geodesic. Furthermore, if the probability distribution function is continuous, geodesics are unique with probability 1. In Section 2, we show that both of these properties have a topological version, which also holds in a large set:

\begin{theorem}
In a residual set of $\Omega$, there exists a geodesic between any two lattice points. Furthermore, if $A$ has no isolated points then in a residual set of $\Omega$ these geodesics are unique.
\end{theorem}

After this, in Section 3, we turn our attention to infinite geodesics, which are self-avoiding paths of infinitely many edges such that each of their finite subpaths are finite geodesics. We distinguish two types of infinite geodesics: the ones indexed by $\mathbb{N}$, informally which are infinite in only one direction, and the ones indexed by $\mathbb{Z}$, informally which are infinite in both directions. We call the former ones geodesic rays while the latter ones geodesic lines.

In the probabilistic setup, one might easily check by K\H{o}nig's lemma the almost sure existence of a geodesic ray, using that with probability 1 there is a geodesic between any two points. In the topological setup, we can use the same argument to prove the same in a residual set of $\Omega$. Namely, denote the first coordinate vector by $\xi_1$ in $\mathbb{R}^d$ and observe a finite geodesic from $0$ to $n\xi_1$ for $n=1,2,...$. As there are finitely many edges having the origin as one of its endpoints, there are infinitely many of these paths which start with the same edge, then there are infinitely many of them which continue with the same edge, etc. This way one might verify the existence of a geodesic ray. Now it is a natural question whether there are more distinct geodesic rays, where by distinct we mean that they share only finitely many edges. In the probabilistic setup it is conjectured that for continuous distributions there are infinitely many of them with probability 1. For d=2 and a certain class of distribution functions this claim was verified in \cite{AD}. In the following two theorems we will prove that in the topological setup we have a totally different phenomenon.

\begin{theorem} If $\sup A>5 \inf A$ then in a residual set of $\Omega$ there is no more than one geodesic ray in $\mathbb{Z}^d$. \end{theorem}

\begin{theorem} For arbitrary $A$, in a residual set of $\Omega$ there exists only a bounded number of distinct geodesic rays in $\mathbb{Z}^d$, more precisely, there are no more than $4d^2$ distinct geodesic rays. \end{theorem}

In Section 4, we revisit an old basic result of first passage percolation, that is the existence of the time constants. Precisely, if we consider any vector $x$, then under mild conditions on the distribution, the function $\frac{T(0,tx)}{t}$ has an almost sure limit in $\infty$ which is usually denoted by $\mu(x)$. One may wonder if it also holds in a large subset of $\Omega$ in the topological setup. The following theorem shows the converse. Throughout the paper, for a vector $x\in\mathbb{R}^d$ we denote by $|x|$ the $\ell_1$ norm of $x$, that is the sum of the absolute values of the coordinates of $x$.

\begin{theorem}
Fix any nonzero vector $x$. In a residual subset of $\Omega$, for any $\lambda$ with
\begin{displaymath}
\inf A \leq \lambda \leq \sup A
\end{displaymath}
\noindent there exists a sequence $(\mu_k)_{k=1}^{\infty}$ with $\mu_k\to\infty$ such that
\begin{displaymath}
\lim_{k\to\infty}\frac{T(0,\mu_k x)}{\mu_k |x|}=\lambda.
\end{displaymath}
\end{theorem}

A related fundamental result is the Cox--Durrett shape theorem. Let us denote by $B(t)$ the ball of radius $t$ centered at the origin in the pseudometric $T$, that is the subset of $\mathbb{R}^d$ we might reach from the origin in time $t$. A truly interesting result of the theory (see \cite{CD}) is that there exists a so-called limit shape $B_\mu$, which has the property that as $t$ tends to infinity, with probability one $\frac{B(t)}{t}$ tends to $B_\mu$ in some sense. Various works can be found in the literature based on this theorem about the speed of this convergence for example. We might ask if a similar statement holds in a residual set in the topological setup. Our next theorem points out it is quite far from the truth under certain, sadly nontrivial conditions on $A$. Let us denote by $D_r$ the $\ell_1$ closed ball of radius $r$ centered at 0, and let $\mathcal{K}_A^d$ be the set of connected compact sets in $\mathbb{R}^d$ satisfying
\begin{displaymath}
D_\frac{1}{\sup A} \subseteq K \subseteq D_\frac{1}{\inf A},
\end{displaymath}
\noindent where the leftmost set is replaced by $\{0\}$ if $\sup A =\infty$, and the rightmost set is replaced by $\mathbb{R}^d$ if $\inf A = 0$. Furthermore, we say that $K\in\mathcal{P}_A^d$ if $K\in\mathcal{K}_A^d$, and there exists $\alpha_K>0$ such that for each $x\in{K}$ there is a "topological path" in $K$ of $\ell_1$-length at most $\frac{1}{\inf A}-\alpha_K$ from $0$ to $x$. (From now on, we use the terms path and topological path in order to clearly distinguish paths in graph theoretical sense and paths in topological sense.) Its closure in $\mathcal{K}_A^d$ with respect to the Hausdorff metric is simply denoted by $\overline{\mathcal{P}_A^d}$. In Section 5, we prove the following:

\begin{theorem}
Assume that $\inf A =0$, or $\sup A=\infty$. Then in a residual subset of $\Omega$ for any $K\in\overline{\mathcal{P}_A^d}$ there exists a sequence $(t_n)_{n=1}^{\infty}$ which tends to infinity and
\begin{displaymath}
\frac{B(t_n)}{t_n} \to K
\end{displaymath}
\noindent in the Hausdorff metric.
\end{theorem}

\section{Finite geodesics}

At the beginning of this section, it might be useful to say a few words about what happens if we allow negative passage times. In this case, it is plain to see that apart from a nowhere dense set of $\Omega$, the passage time between any two points would be $-\infty$. To verify this, we declare at this point how we will think about the topology on $\Omega$. The most convenient way for us is to consider cylinder sets as the basis of the topology, that is the basis sets are of the form
\begin{displaymath}
U=\times_{e\in E}U_e,
\end{displaymath}
\noindent where each $U_e$ is open in $A$ and with at most finitely many exceptions $U_e=A$. We say that $U_e$ is the projection of $U$ to the edge $e$.

Using this, we can easily verify our previous claim. We need that if $U$ is a nontrivial open set of $\Omega$, then there exists a nontrivial open set $V\subseteq{U}$ such that on $V$, the passage time between any two points is $-\infty$. It clearly suffices to show this for a cylinder set $U$, which is rather straightforward: as there are only finitely many edges for which $U$ has nontrivial projection, we can choose an edge $e$ with trivial projection. Then we define $V$ to have the same projections everywhere as $U$, except for $e$, where the projection contains only negative values. Then in any configuration in $V$, the passage time between two lattice points $x,y$ is $-\infty$: indeed, we can take paths of arbitratily low passage time by going to one of the vertices of $e$ from $x$ on a fixed route, then go along it back and forth as many times as we wish, and then finally go to $y$ on a fixed route. The first and the last part of this path has a fixed passage time in a given configuration, while the middle term can be arbitrarily low. Thus apart from a nowhere dense set of $\Omega$, the passage time between any two lattice points is $-\infty$ indeed, and it quickly yields the same for any two points.

One may wonder what happens if we allow negative passage times, but we only permit self-avoiding paths, except for that the starting and the ending point of a path may coincide. This restriction clearly rules out our previous argument, however, we might expect that passage times are still $-\infty$ in a considerably large set if $d\geq{2}$. (If $d=1$, we have only one possible path between any two vertices, thus in a reasonably small open set there are vertices whose $T$-distance is quite well determined. As a consequence, it is something we are not interested in.) The following theorem shows that the above expectation is true.

\begin{theorem} Suppose that $A$ contains a negative value, and we define $T(x,y)$ by considering the infimum only for the paths which might contain only their starting point and endpoint twice. Then in a residual subset of $\Omega$ we have $T(x,y)=-\infty$ for any two points $x,y$. \end{theorem}

\begin{proof} As the passage time between any two non-lattice points equals the passage time between certain lattice points, it suffices to prove that in a residual subset of $\Omega$ we have $T(x,y)=-\infty$ for any two lattice points $x,y$. Denote the subset where this holds by $S$. Furthermore, let us denote the set of configurations satisfying $T(x,y)<-n$ for some $n\in\mathbb{N}$ by $S(x,y,n)$. Using this notation, we have
\begin{displaymath}
S=\bigcap_{x,y\in\mathbb{Z}^d,\text{ } n\in\mathbb{N}}S(x,y,n),
\end{displaymath}
\noindent which is a countable intersection. As a consequence, it suffices to prove that each $S(x,y,n)$ is residual. By definition, this is equivalent to $Q(x,y,n)=\Omega\setminus S(x,y,n)$ is meager. In fact, we will prove that $Q(x,y,n)$ is nowhere dense. Fix $U$ to be a cylinder set. Let us denote the set of edges belonging to nontrivial projections of $U$ by $E_U=\{e_1, e_2, ..., e_k\}$. By shrinking the projections $U_{e_1}, ..., U_{e_k}$, we can achieve that all of them are bounded in $\mathbb{R}$. Denote these new projections by $U'_{e_i}$, $i=1,...,k$, and the cylinder set defined by them by $U'$. Then for any configuration in $U'$, the sum of passage times over the edges $e_1,...,e_k$ is bounded by a constant $C$. 

Let $a\in{A}$ be negative. Note that we might construct a self-avoiding path from $x$ to $y$ of arbitrarily large $\ell^1$ length, or in other words, of arbitrarily large number of edges. Indeed, we can go arbitrarily far along the direction of one axis, and then if we forget about these edges, the remaining graph is still connected as $d\geq{2}$. Thus we might consider a path $\Gamma$ from $x$ to $y$ with length $m$ for large enough $m$. We determine $m$ later. Now we define $V\subseteq{U'}$ to have the same projections as $U'$, except for the edges in $\Gamma\setminus{E_U}$: here we define the projections to be a subset of $(-\infty,\frac{a}{2})$. In $V$, we can bound the passage time of $\Gamma$ as it follows:
\begin{displaymath}
\sum_{e\in\Gamma}t(e)=\sum_{e\in\Gamma\setminus E_U}t(e)+\sum_{e\in\Gamma\cap E_U}t(e)\leq (m-k)\frac{a}{2}+C<-n,
\end{displaymath}
\noindent if $m$ is large enough, as $\frac{a}{2}<0$ and $k,C$ are fixed. Thus the configurations in $V$ cannot be in $Q(x,y,n)$, yielding $Q(x,y,n)$ is nowhere dense, which is what we wanted to prove. \end{proof}

In the sequel, we will return to the case when $A$ contains only nonnegative numbers. First, we prove Theorem 1.1.

\begin{proof}[Proof of Theorem 1.1] Consider the first statement. We will prove that for given $x,y\in\mathbb{Z}^d$, apart from a nowhere dense set of $\Omega$, there exists a geodesic between $x$ and $y$. As there are countably many such pairs, it would be sufficient. The idea of the proof is that typicially the paths with reasonably low passage times lie in a bounded set containing $x$ and $y$, thus if we are interested in $T(x,y)$, we have to consider only finitely many paths, hence the infimum is the minimum. 

To verify our claim, fix lattice points $x,y$ and a cylinder set $U$. Let us denote the set of edges belonging to nontrivial projections of $U$ by $E_U=\{e_1, e_2, ..., e_k\}$. As in the proof of Theorem 2.1, we can construct a smaller cylinder set by shrinking the projections $U_{e_1}, ..., U_{e_k}$, such that all of these projections are bounded in $\mathbb{R}$. We denote again these new projections by $U'_{e_i}$, $i=1,...,k$, and the cylinder set defined by them by $U'$. Then for any configuration in $U'$, the sum of passage times over the edges $e_1,...,e_k$ is bounded by a constant $C$. Choose an $a \in A$ such that $a>0$. We will fix an $n\in\mathbb{N}$ later. Consider all the edges with $\ell^1$ distance at most $n$ from $x$. Denote their set by $E^*$. If $n$ is large enough, there is an optimal $\ell^1$ path from $x$ to $y$ using only edges in $E^*$. We will define $V\subseteq{U'}$ as a cylinder set which has nontrivial projections to the edges in $E_U\cup E^*$. Concerning the edges in $E_U$, we define $V$ to have the same projections as $U'$, meanwhile for the edges in $E^*\setminus E_U$, we require that the projections of $V$ equal $(a-\varepsilon,a+\varepsilon)\cap{A}$, where $0<\varepsilon<a$. Consider now any configuration in $V$, and take a path $\Gamma$ from $x$ to $y$ with $\ell_1$-length $|x-y|$, using only edges in $E^*$. Then its passage time is at most $Ck+|x-y|(a+\varepsilon)=C_1$, a constant independent from the actual configuration in $V$. Meanwhile if we consider any path from $x$ to $y$ which uses an edge which is not in $E^*$ its passage time is at least $(n-k)(a-\varepsilon)$ for any configuration in $V$, as it has to use at least $n$ edges to leave $E^*$, and apart from the at most $k$ edges in $E_U$ they have passage time at least $a-\varepsilon$. However, for large enough $n$, this passage time eventually surpasses $C_1$. As a consequence, if we define $E^*$ and then $V$ using this $n$, we will know that for any configuration in $V$, the paths leaving $E^*$ have passage times higher than the passage time of $\Gamma$. Hence in the definition of $T(x,y)$, we have to consider only the paths connecting $x,y$ which use only edges in $E^*$. There are finitely many of them, thus in fact the infimum is the minimum, yielding we have a geodesic between $x,y$ in $V$. As a consequence, as we claimed, there exists a geodesic between $x$ and $y$ apart from a nowhere dense set of $\Omega$. 

What remains to prove is the uniqueness part of Theorem 1.1. It suffices to prove that for given lattice points $x,y$, apart from a nowhere dense set there is a unique geodesic between $x$ and $y$. Fix a cylinder set $U$. By the previous argument, we know that there exists of a cylinder set $V\subseteq{U}$ such that in $V$, there is a geodesic between $x$ and $y$. We will shrink this cylinder set further to arrive at a cylinder set $W$ in which there is always a unique geodesic between $x$ and $y$. In order to do so, define for each path $\Gamma$ connecting $x$ and $y$ the number $\tau(\Gamma,V)$ as the infimum of passage times of $\Gamma$ for configurations in $V$. Let $\tau(V)=\inf_\Gamma\tau(\Gamma,V)$. By the definition of $V$, this is determined by finitely many paths from $x$ to $y$ in fact, as for any configuration in $V$, the too long paths have too large passage times. Thus $\tau(V)$ equals a minimum, and in the sequel, we might focus only on these paths. Let $\Gamma_0$ be one of the paths for which $\tau(\Gamma_0,V)=\tau(V)$. It would be nice to have a unique path with this property: from this point, the construction of $W$ would be more or less straightforward. We claim that for an appropriate $V'\subseteq{V}$ we can have $\tau(\Gamma_0,V')=\tau(V')=\tau(V)$ while for any $\Gamma\neq\Gamma_0$ we have $\tau(\Gamma,V')>\tau(V')$. Indeed, if we define $V'$ to have the same projections as $V$ to the edges contained by $\Gamma_0$, we immediately have our first requirement. Furhermore, if $\Gamma\neq\Gamma_0$ with $\tau(\Gamma,V)=\tau(V)$, there is at least one edge $e\in\Gamma\setminus\Gamma_0$. We will shrink the projection to this edge: as $A$ has no isolated points, we can choose some nonempty $V'(e)\subseteq V(e)\cap A$ with higher infimum than $\inf V(e)$, which results in $\tau(\Gamma,V')>\tau(\Gamma,V)\geq\tau(V')$. Repeating the same step for each $\Gamma\neq\Gamma_0$ with $\tau(\Gamma,V)=\tau(V)$ (which means only finitely many steps) we obtain some $V'$ with the above property.

In the final step we will only shrink the projections of $V'$ to the edges in $\Gamma_0$. As $\tau(\Gamma,V')>\tau(\Gamma_0,V')$ for any $\Gamma_0\neq\Gamma$, and there are only finitely many paths we are interested in by now, for some $\varepsilon>0$ we have $\tau(\Gamma,V')>\tau(\Gamma_0,V')+\varepsilon$. We will shrink the projections of $V'$ to the edges in $\Gamma_0$ based on this bound. Namely, if $\Gamma_0$ contains the edges $e_1', ..., e_m'$, and the infimum of $V'(e_i')$ is $a_i$, we will define $W(e_i)$ as $\left(a_i,a_i+\frac{\varepsilon}{m}\right)\cap A$. Then as
\begin{displaymath}
\tau(\Gamma_0,V')\geq\sum_{i=1}^{m}a_i,
\end{displaymath}
\noindent we have that for any configuration $\omega\in W$ the passage time of $\Gamma_0$ is at most
\begin{displaymath}
\tau_{\omega}(\Gamma_0)\leq\sum_{i=1}^{m}\left(a_i+\frac{\varepsilon}{m}\right)\leq\tau(\Gamma_0,V')+\epsilon<\tau(\Gamma,V')\leq\tau_\omega(\Gamma)
\end{displaymath}
\noindent for any $\Gamma\neq\Gamma_0$, as a configuration in $W$ is also in $V'$, hence $\tau(\Gamma,V')\leq\tau_\omega(\Gamma)$. Thus for any configuration in $W$, the unique geodesic from $x$ to $y$ is $\Gamma_0$. This concludes the proof. \end{proof}

\begin{megj} In the proof we clearly used that $A$ has no isolated points to be able to nontrivially shrink open sets in $A$. By a similar argument, one can quickly check that if $A$ has an isolated point $a$, then for any two lattice points $x,y$ such that the line segment $[x,y]$ is not parallel to any of the coordinate axis (i.e. there are multiple optimal $\ell_1$ paths from $x$ to $y$), there exists a cylinder set $U$ such that for any configuration in $U$ there are multiple geodesics from $x$ to $y$. Indeed, we can define $U$ to have projections containing only $a$ to the set of edges within a given large $\ell_1$ distance to $[x,y]$, similarly to the definition of $V$ in the previous proof. Then it is easy to see that the geodesics between $x$ and $y$ are precisely the optimal $\ell_1$ paths, of which there are more than one. \end{megj}

\section{Infinite geodesics}

\begin{proof}[Proof of Theorem 1.2] First we will prove that if $x$ is a given lattice point then apart from a nowhere dense set of $\Omega$ there is no more than one geodesic ray starting from $x$. Clearly it suffices to prove this claim concerning geodesic rays starting from the origin. Let $F(0)$ denote the set of configurations in which there are at least two distinct geodesic rays starting from the origin. Then $F(0)=\bigcup_{m=1}^{\infty}F_m(0)$ where $F_m(0)$ stands for the set of configurations in which there are at least two distinct geodesics starting from the origin such that they have at most $m$ edges in common. We claim that for any $m$ we have that $F_m(0)$ is a nowhere dense set in $\Omega$, which would verify our preliminary statement about the meagerness of $F(0)$.

As usual, fix $U$ to be a cylinder set, and denote the set of edges belonging to nontrivial projections of $U$ by $E_U=\{e_1, e_2, ..., e_k\}$. As in the proof of Theorem 2.1, we can construct a smaller cylinder set by shrinking the projections $U_{e_1}, ..., U_{e_k}$, such that all of these projections are bounded in $\mathbb{R}$. We denote again these new projections by $U'_{e_i}$, $i=1,...,k$, and the cylinder set defined by them by $U'$. Then for any configuration in $U'$, the sum of passage times over the edges $e_1,...,e_k$ is bounded by a constant $C$. Consider now the hypercubes $K_1=[-p_1,p_1]^d$ and $K_2=[-p_2,p_2]^d$, where we choose $p_1\in\mathbb{N}$ such that the interior of $K_1$ contains all the edges of $E_U$, while the precise value of $p_2>p_1$ is to be determined later. Let us denote the set of edges in $K_2$ which are not in the interior of $K_1$ by $E^*$. We will define $V\subseteq{U'}$ as a cylinder set which has nontrivial projections to the edges in $E_U\cup E^*$. The idea is the following: for the configurations in $V$ we would like to have essentially one (and the same) geodesic from the boundary $\partial K_1$ to the boundary $\partial K_2$, for example the line segment connecting $p_1\xi_1$ and $p_2\xi_1$. (We recall that $\xi_1$ is the first coordinate vector in $\mathbb{R}^d$.) By this we mean that for any lattice points $x_1\in \partial K_1$ and $x_2 \in\partial K_2$, a geodesic $\Gamma$ from $x_1$ to $x_2$ eventually arrives in $p_1\xi_1$, and then it goes along the line segment $[p_1\xi_1,p_2\xi_1]$. If we could achieve this, we would be done: as any geodesic ray starting from the origin eventually leaves $K_1$ and $K_2$, and a geodesic ray is a geodesic between any two of its points, the previous properties would guarantee that any geodesic ray starting from the origin would go along the line segment $[p_1\xi_1,p_2\xi_1]$. However, that would mean that our configuration cannot be in $F_m(0)$ for $p_2-p_1>m$ as there would not exist at least two distinct geodesics starting from the origin such that they have at most $m$ edges in common.

Let us make the above argument rigorous. Let $\varepsilon>0$ be such that $5(\inf A + \varepsilon)<\sup A - \varepsilon$ still holds. We would like to have small passage times on the edges of $\partial K_1$, $\partial K_2$, and along the line segment $[p_1\xi_1,p_2\xi_1]$ to guarantee a path with considerably low passage time between any two points of $\partial K_1$ and $\partial K_2$. We call these edges cheap. Meanwhile on other edges between the two boundaries (e.g. the expensive edges) we would like to have as large passage times as possible. Thus for every cheap edge $e$ we define 
\begin{displaymath}
V_e=[\inf A, \inf A + \varepsilon_e ) \cap A,
\end{displaymath}
\noindent where the $\varepsilon_e$s are defined such that their sum for cheap edges is at most $\varepsilon$. Meanwhile for any expensive edge we define
\begin{displaymath}
V_e=(a-\varepsilon_e, a+ \varepsilon_e ) \cap A,
\end{displaymath}
\noindent where $a\in{A}$ is chosen such that $a>5(\inf A + \varepsilon)$ and again the $\varepsilon_e$s are defined such that their sum for expensive edges is at most $\varepsilon$. By this, we have formally defined $V$. Now consider any configuration in $V$. Our aim is to prove that if $\Gamma$ is a path from some $x_1\in\partial K_1$ to some $x_2\in\partial K_2$, and it does not contain the line segment $[p_1\xi_1,p_2\xi_1]$, then it cannot be a geodesic. Proceeding towards a contradiction, assume the existence of $x_1\in\partial K_1$ and $x_2\in\partial K_2$ such that the geodesic $\Gamma$ from $x_1$ to $x_2$ does not contain the line segment $[p_1\xi_1,p_2\xi_1]$. As $\Gamma$ is a geodesic between any two points of it, we might suppose that $x_1$ is its only point on $\partial K_1$ and $x_2$ is its only point on $\partial K_2$: otherwise we might replace $\Gamma$ by a subpath of it. Consider first the case when $\Gamma$ does not share any edge with the line segment $[p_1\xi_1,p_2\xi_1]$. Then $\Gamma$ must contain at least $|x_2-x_1|$ expensive edges, which gives the following bound for the passage time:
\begin{displaymath}
\tau(\Gamma)\geq 5|x_2-x_1|(\inf A + \varepsilon)-\varepsilon.
\end{displaymath}
\noindent We will construct another path $\Gamma_0$ from $x_1$ to $x_2$, which uses only cheap edges. (See Figure 1 for $d=2$.) First, we go from $x_1$ to $p_1\xi_1$ on $\partial K_1$ using the shortest possible way in $\ell_1$. It is simple to see that this part requires at most $p_1+2dp_1$ edges: if needed, using a segment of length $p_1$ we might arrive at a facet which is neighboring to the one containing $p_1\xi_1$, and then we do not need more edges than the $\ell_1$ diameter of $K_1$, which is $2dp_1$. Now we proceed to $p_2\xi_1$ along the line segment $[p_1\xi_1,p_2\xi_1]$, this step clearly requires $p_2-p_1$ edges. Finally, we go to $x_2$ on $\partial K_2$ once again using the shortest possible way in $\ell_1$.
\begin{figure}[h!]
  \includegraphics[width=250pt]{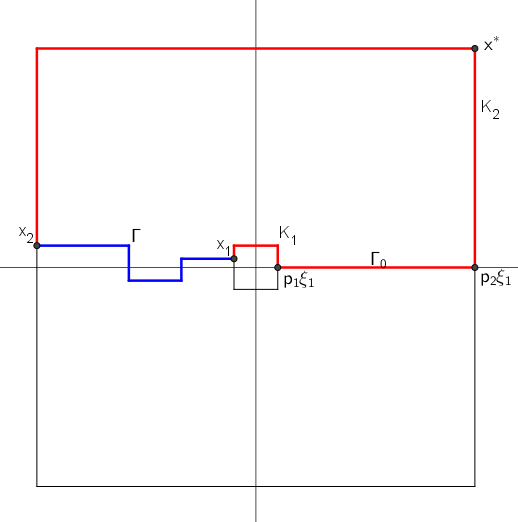}
  \caption{The case when $\Gamma$ and $[p_1\xi_1,p_2\xi_1]$ have no common edges for $d=2$}
  \label{abra1}
\end{figure}

\noindent The number of edges needed in this final step can be bounded the following way: if needed, using a segment of length $p_2$ we might arrive at a facet which is neighboring to the one containing $x_2$ in a point $x^*$. From here, we can get to $x_2$ by an optimal $\ell_1$ path whose are edges contained by $\partial K_2$: indeed, on the boundary of a cube any two points lying on neighboring facets are connected by such an optimal path, as it is unnecessary to take steps in opposite directions. Thus in this step, we need at most $p_2+|x_2-x^*|$ edges, and $|x_2-x^*|$ can be estimated by a simple triangle inequality using our previous remarks:
\begin{equation}
\begin{split}
|x_2-x^*| & \leq |x_2-x_1|+|x_1-p_1\xi_1|+|p_1\xi_1-p_2\xi_1|+|p_2\xi_1-x^*| \\ & \leq |x_2-x_1|+ 2dp_1+(p_2-p_1)+p_2.
\end{split}
\end{equation}
\noindent Thus by counting the edges in each part of $\Gamma_0$ we get an estimate for its $\ell_1$-length:
\begin{equation}
\begin{split}
|\Gamma_0| & \leq \left(2d+1\right)p_1+\left(p_2-p_1\right)+p_2+\left(|x_2-x_1|+p_2+(p_2-p_1)+2dp_1\right) \\ & \leq 4p_2+4dp_1+|x_2-x_1|.
\end{split}
\end{equation}
\noindent Given that $\Gamma_0$ only uses cheap edges, it also yields a bound for its passage time:
\begin{displaymath}
\tau(\Gamma_0)\leq (4p_2+4dp_1+|x_2-x_1|)\inf A +\varepsilon.
\end{displaymath}
\noindent Comparing the bounds on $\tau(\Gamma)$ and $\tau(\Gamma_0)$ we see that the desired inequality $\tau(\Gamma_0)<\tau(\Gamma)$ necessarily holds if 
\begin{displaymath}
(4p_2+4dp_1+|x_2-x_1|)\inf A +\varepsilon< 5|x_2-x_1|(\inf A + \varepsilon)-\varepsilon.
\end{displaymath}
\noindent However, as $|x_2-x_1|\geq{p_2-p_1}$, it is easy to see that this holds if $p_2$ is large enough. Thus we ruled out the possibility of the existence of a geodesic from $\partial K_1$ to $\partial K_2$ which does not even share edges with $[p_1\xi_1,p_2\xi_1]$.

The cases where $\Gamma$ contains some, but not all of the edges of $[p_1\xi_1,p_2\xi_1]$ can be handled similarly. In order to do this, denote the last point of $\Gamma$ in $\partial K_1$ by $y_1$. By the previous case, if $\Gamma$ is a geodesic from $\partial K_1$ to $\partial K_2$, it must contain a point of the line segment $[p_1\xi_1,p_2\xi_1]$ after passing through $y_1$. Denote the first such point by $z_1$. Assume that these points are distinct. Between these points $\Gamma$ only uses expensive edges. However, by geometry, between $y_1$ and $z_1$ there exists a path not longer in $\ell_1$ using only cheap edges, which is necessarily cheaper than the original path which only used expensive edges. This argument shows that for such a geodesic $\Gamma$, we must have $y_1=z_1=p_1\xi_1$.

Furthermore, let us denote by $z_2$ the last point of $\Gamma$ on $[p_1\xi_1,p_2\xi_1]$ after leaving $\partial K_1$, and by $y_2$ the first point of $\Gamma$ on $\partial K_2$. We claim that if $\Gamma$ is a geodesic, then $y_2=z_2$. Assume $y_2 \neq {z_2}$, that is, $\Gamma$ uses expensive edges after hitting $z_2$. The case when $y_2$ lies on the same facet of $K_2$ as $p_2\xi_1$ can be ruled out by the same geometric argument we used just before: in this case there exists an optimal $\ell_1$ path using only cheap edges, which is necessarily cheaper than any path using expensive edges. Finally, if $y_2$ lies on another facet of $K_2$, then $\Gamma$ clearly needs at least $|y_2-z_2|$ expensive edges to reach it from $z_2$, where $|y_2-z_2|\geq p_2-p_1$. However, the argument we used to rule out the case when $\Gamma$ does not share any edge with $[p_1\xi_1,p_2\xi_1]$ was essentially based on this inequality. Indeed, we can construct basically the same $\Gamma_0$ from $z_2$ to $y_2$, which is even more simple as the first two parts can be replaced by the segment $[z_2,p_2\xi_2]$, and then we might use the same estimates. It proves $y_2=z_2$.

Thus if $p_2$ is chosen to be large enough for any configuration in $V$ we have that any geodesic ray starting from the origin contains the line segment $[p_1\xi_1,p_2\xi_1]$, which verifies that $F_m(0)$ is nowhere dense. Consequently, we obtain that in a residual set of $\Omega$ there is no more than one geodesic ray starting from a given lattice point which concludes the proof of our weaker statement.

Now let $F\subseteq{\Omega}$ be the set of configurations in which there are at least two geodesic rays. Denote by $F(x)$ the set of configurations in which there are at least two distinct geodesic rays starting from $x$, and by $F_m$ the set of configurations in which there exist two disjoint geodesic rays starting from the cube $[-m,m]^d$. Then 
\begin{displaymath}
F=\left(\bigcup_{x\in\mathbb{Z}^d}F(x)\right)\cup\left(\bigcup_{m=1}^{\infty}F_m\right)
\end{displaymath}
\noindent clearly holds: if there exist at least two geodesic rays they are either disjoint or meet at some point $x$, and in the latter case we have two geodesic rays starting from $x$ if we forget about the initial parts of these geodesics. Furthermore, we know that each of sets $F(x)$ are meager. Thus if we could obtain that each $F_m$ is nowhere dense, that would conclude the proof. However, having seen the proof of the first part we do not have a difficult task as we can basically repeat that argument. Indeed, in that proof we showed that for a given cylinder set $U$ one can construct cubes $K_1,K_2$ and another cylinder set $V\subseteq{U}$ such that for configurations in $V$ any geodesic from $\partial K_1$ to $\partial K_2$ goes along the line segment $[p_1\xi_1,p_2\xi_1]$. Thus if we choose $p_1>m$ during the construction we will obtain that none of the configurations in $F_m$ can appear in $V$ as in $V$ there cannot be two disjoint geodesic rays starting from $[-m,m]^d$, given they all meet in $p_1\xi_1$. Thus $F_m$ is nowhere dense indeed, which concludes the proof of the theorem. \end{proof}

Theorem 1.2 has the following obvious corollary about geodesic lines, as a geodesic line can be considered as the union of two distinct geodesic rays:

\begin{kov} If $\sup A>5 \inf A$ then in a residual set of $\Omega$ there exists no geodesic line in $\mathbb{Z}^d$. \end{kov}

\begin{proof}[Proof of Theorem 1.3] In this case, we will also consider first geodesic rays starting from the origin: we will prove that in a residual set of $\Omega$ there cannot be more than $2d$ distinct geodesic rays starting from the origin. Let $F(0)$ denote now the set of configurations in which there are at least $2d+1$ distinct geodesic rays starting from the origin. We decompose $F(0)$ as $\bigcup_{m=1}^{\infty}F_m(0)$ where $F_m(0)$ stands for the set of configurations in which there are at least $2d+1$ distinct geodesic rays starting from the origin and any two have at most $m$ edges in common. Proving that $F_m$ is nowhere dense for each $m$ would verify our first claim. To check this, we will use a similar machinery as in the proof of Theorem 1.2. Let $U$ be fixed cylinder set and define the hypercubes $K_1$ and $K_2$ as back there. Our goal is to have a control over geodesics from $\partial K_1$ to $\partial K_2$, but instead of having essentially one cheap path between the two boundaries, now we can have $2d$. Informally, the idea is quite straightforward: we will have cheap edges on the boundaries and along the line segments $[p_1\xi_i,p_2\xi_i]$ and $[-p_1\xi_i,-p_2\xi_i]$ for $i=1,2,...,d$, (where $\xi_i$ denotes the $i$-th coordinate vector), while we will have expensive edges on the remaining edges between the two boundaries.
Formally, we fix some $\varepsilon>0$ which is at most the half of the diameter of $A$, and for every cheap edge $e$ we define 
\begin{displaymath}
V_e=[\inf A, \inf A + \varepsilon_e ) \cap A,
\end{displaymath}
\noindent where the $\varepsilon_e$s sum for cheap edges is at most $\varepsilon$. Furthermore we fix some $\lambda>1$ such that there exists $a\in A$ satisfying $a>\lambda(\inf A + \varepsilon)$. For any expensive edge we define
\begin{displaymath}
V_e=(a-\varepsilon_e, a+ \varepsilon_e ) \cap A,
\end{displaymath}
\noindent where the $\varepsilon_e$s sum for expensive edges is at most $\varepsilon$. Furthermore, for the sake of brevity we introduce the notation $I_i=[p_1\xi_i,p_2\xi_i]$ and $-I_i=[-p_1\xi_i,-p_2\xi_i]$.

We claim that in $V$ for any $x_1\in\partial K_1$ and $x_2 \in\partial K_2$, and any path $\Gamma$ from $x_1$ to $x_2$ which contains none of the segments $I_i$, or $-I_i$, there exists a cheaper path which contains one of them. Proceeding towards a contradiction, assume the existence of $x_1,x_2,\Gamma$ such that there is no such a cheaper path in a certain configuration. Consider such $\Gamma$ with minimal $\ell_1$-length. Then we obviously have that $x_1$ is the only point of $\Gamma$ on $\partial K_1$. Indeed, assume for example that $x_1'$ is another point of $\Gamma\cap\partial K_1$. Denote the subpath of $\Gamma$ from $x_1$ to $x_1'$ by $\Gamma_1$, and the subpath from $x_1'$ to $x_2$ by $\Gamma_2$. Now if there would exist a cheaper path $\Gamma_2'$ from $x_1'$ to $x_2$ containing one of the segments $I_i$ or $-I_i$, we would immediately have that $\Gamma_1\cup\Gamma_2'$ is cheaper than $\Gamma$, and contains one of these segments, a contradiction. Thus the path $\Gamma_2$ connecting $x_1'$ and $x_2$ is also a path with the property that there is no cheaper path containing any of the segments $I_i$ or $-I_i$, and its $\ell_1$-length is smaller than the $\ell_1$-length of $\Gamma$. It cannot happen by the definition of $\Gamma$, thus we indeed have that $x_1$ is the only point of $\Gamma$ on $\partial K_1$. Similarly one can show that $x_2$ is the only point of $\Gamma$ on $\partial K_2$. 

First, let us consider the case when $\Gamma$ does not use any cheap edge. In this case, as $\Gamma$ cannot enter $K_1$, it uses at least $|x_2-x_1|$ expensive edges, yielding 
\begin{displaymath}
\tau(\Gamma)\geq \lambda|x_2-x_1|(\inf A + \varepsilon)-\varepsilon.
\end{displaymath}
\noindent Without limiting generality, for now we may assume that $x_2$ lies on the same facet of $K_2$ as $p_2\xi_1$. Compare $\Gamma$ to the following path $\Gamma_0$ from $x_1$ to $x_2$: first, we go from $x_1$ to $p_1\xi_1$ on $\partial K_1$ using the shortest possible way $\ell_1$, then we proceed to $p_2\xi_1$ along the line segment $[p_1\xi_1,p_2\xi_1]$, finally we get to $x_2$ on $\partial K_2$ once again using the shortest possible way in $\ell_1$. 
\begin{figure}[h!]
  \includegraphics[width=250pt]{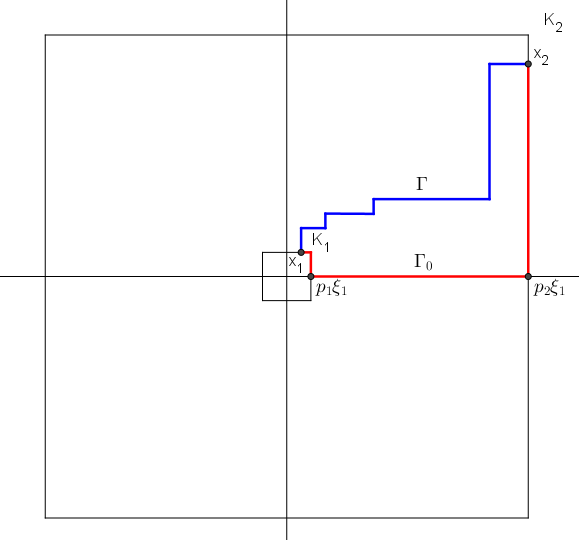}
  \caption{The case when $\Gamma$ and the line segments $I_i, -I_i$ have no common edges for $d=2$}
  \label{abra2}
\end{figure}

Then this path contains only cheap edges. Furthermore, we can get from $x_1$ to $p_1\xi_1$ using at most $p_1+2dp_1$ edges (as in the proof of Theorem 1.2), and afterwards we use at most $4p_1+|x_2-x_1|$ edges by a simple triangle inequality. Thus we have
\begin{displaymath}
\tau(\Gamma_0)\leq (|x_2-x_1|+(5+2d)p_1)\inf A +\varepsilon.
\end{displaymath}
\noindent Comparing the bounds on $\tau(\Gamma)$ and $\tau(\Gamma_0)$ we see that the desired inequality $\tau(\Gamma_0)<\tau(\Gamma)$ necessarily holds if
\begin{displaymath}
(5+2d)p_1 \inf A +2\varepsilon< \varepsilon|x_2-x_1|,
\end{displaymath}
\noindent which trivially holds if $p_2$ is large enough, since $|x_2-x_1|\geq{p_2-p_1}$. Thus we ruled out the case when $\Gamma$ does not even share edges with any of the line segments $I_i$ or $-I_i$.

Let us assume now that $\Gamma$ contains some edges of one of the line segments $I_i$ or $-I_i$. Denote the first point of $\Gamma$ on one of these line segments by $z_1$. By the same geometric argument as in the previous proof, between $x_1$ and $z_1$ there exists an optimal $\ell_1$ path using only cheap edges, which is necessarily cheaper than any path using expensive edges. Thus we have $x_1=z_1=p_1\xi_i$ or $x_1=z_1=-p_1\xi_i$ for some $i=1,2,...,d$, and the first edge of $\Gamma$ necessarily lies on $I_i$ or $-I_i$. By symmetry and without limiting generality, we can assume it lies on $I_1$. In this case, by the same virtue we can deduce that $\Gamma$ does not contain any points of any line segment $I_i$ or $-I_i$ distinct from $I_1$. Indeed, if that would be the case for some $p\in I_i$ or $p\in -I_i$, then $\Gamma$ would use expensive edges as it has only one point on $\partial K_1$ and $\partial K_2$, and without the edges of these boundaries the $I_i,-I_i$s are pairwise disconnected if we consider cheap edges only. However, to such a point $p$ we have a path $\gamma$ from $p_1\xi_1$ using only cheap edges, which is not longer in $\ell_1$ than any path using expensive edges. Thus $\gamma$ is cheaper than the subpath of $\Gamma$ connecting $p_1\xi_1$ and $p$, a contradiction.

Furthermore, let us denote by $z_2$ the last point of $\Gamma$ on $I_1$ after leaving $z_1$. We claim that if $\Gamma$ is a geodesic, then $x_2=z_2=p_2\xi_1$. Assume $x_2 \neq {z_2}$, that is $\Gamma$ uses expensive edges after hitting $z_2$. The case when $x_2$ lies on the same facet of $K_2$ as $p_2\xi_1$ can be ruled out by the same geometric argument we used just before. 

Assume $x_2$ lies on another facet of $K_2$, first assume that it is a neighboring one. By symmetry, we can assume $x_2$ is on the same facet as $p_2\xi_2$. Let $z_2=q_2\xi_1$, then the geodesic $\Gamma$ from $x_1=p_1\xi_1$ to $x_2$ uses $q_2-p_1$ cheap edges and then at least $|x_2-z_2|$ expensive edges. Thus for the passage time of $\Gamma$ we have
\begin{displaymath}
\tau(\Gamma)\geq (q_2-p_1)\inf A+\lambda|x_2-z_2|(\inf A + \varepsilon)-\varepsilon.
\end{displaymath}
\noindent Compare it to the following path $\Gamma_0$ from $x_1$ to $x_2$ (see Figure 3): we go from $x_1$ to $p_1\xi_2$ using the shortest possible way on $\partial K_1$, then we go along $I_2$, after this we proceed to $z_2+p_2\xi_2$ using the shortest possible way on $\partial K_2$, finally we go to $x_2$ on $\partial K_2$ again.
\begin{figure}[h!]
  \includegraphics[width=250pt]{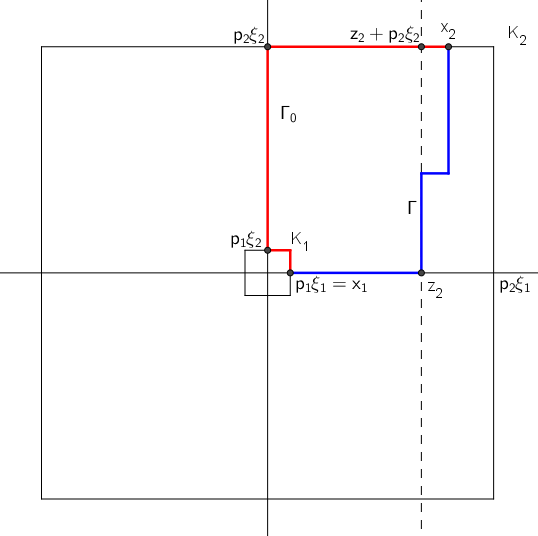}
  \caption{The case when $x_2$ lies on a neighboring facet}
  \label{abra3}
\end{figure}
Then $\Gamma_0$ uses only cheap edges, precisely $(q_2+p_1)+|x_2-z_2|$ of them as one can easily check. Thus for the passage time of $\Gamma_0$, we have the following estimate:
\begin{displaymath}
\tau(\Gamma_0)\leq ((q_2+p_1)+|x_2-z_2|)\inf A+\varepsilon.
\end{displaymath}
\noindent Using the fact that $|x_2-z_2|$ is at least $p_2$, the comparison of these bounds is similar to the already seen ones and one quickly obtains $\tau(\Gamma_0)<\tau(\Gamma)$ for large $p_2$, contradicting the fact that $\Gamma$ is a geodesic.

Finally, if $x_2$ lies on the opposite facet of $K_2$, that is the same facet where $-p_2\xi_1$ can be found, then we can use almost the same estimates after choosing $\Gamma_0$ to be the path from $x_1$ to $x_2$ of cheap edges containing $-I_1$, and arrive at a contradiction the same way. This contradiction concludes the proof of the weaker statement: in a residual set of $\Omega$ there cannot be more than $2d$ distinct geodesic rays starting from the origin.

To obtain the original statement of the theorem, we just borrow the idea of the proof of Theorem 1.2. Let $F\subseteq{\Omega}$ be the set of configurations in which there are at least $4d^2+1$ distinct geodesic rays. Denote by $F(x)$ the set of configurations in which there are at least $2d+1$ distinct geodesic rays starting from $x$, and by $F_m$ the set of configurations in which there exist $4d^2+1$ geodesic rays starting from the cube $[-m,m]^d$ such that any $2d+1$ of them has an empty intersection. Then 
\begin{displaymath}
F\subseteq\left(\bigcup_{x\in\mathbb{Z}^d}F(x)\right)\cup\left(\bigcup_{m=1}^{\infty}F_m\right)
\end{displaymath}
\noindent clearly holds. Furthermore, we know that each of the sets $F(x)$ are meager. Thus if we could obtain that any $F_m$ is nowhere dense, that would conclude the proof. However, having seen the proof of the first part we do not have a difficult task as we can basically copy that argument. Indeed, in that proof we showed that for a given cylinder set $U$ one can construct the cubes $K_1,K_2$ and another cylinder set $V\subseteq{U}$ such that for configurations in $V$ any geodesic from $\partial K_1$ to $\partial K_2$ goes along one of the line segments $I_i$ or $-I_i$. Thus if we choose $p_1>m$ during the construction we will obtain that none of the configurations in $F_m$ can appear in $V$. Indeed, by the pigeonhole principle in $V$ there cannot be $4d^2+1$ geodesic rays starting from $[-m,m]^d$ such that any $2d+1$ of them has an empty intersection, given that all of them uses one of the $2d$ line segments. Thus $F_m$ is nowhere dense indeed, which concludes the proof of the theorem. \end{proof}

By the same machinery one can obtain the following intermediate result by placing cheap edges on the boundaries and for example on $I_1$ and $-I_1$: if $\sup A  > 3 \inf A$, then in a residual set of $\Omega$ there is no more than 2 geodesic rays starting from the origin, and there is no more than 4 geodesic rays altogether. However, after seeing the previous proofs it is not of much interest. What would be more exciting, that is to give a set $A$ such that there are at least two geodesics in a residual set of $\Omega$. Obviously, Theorem 1.2 excludes a lot of possibilites, however, one might hope that if $A$ is sufficiently narrow then there should be two geodesic rays in a large set of $\Omega$, for example one heading somewhat to the direction of $\xi_1$ and another heading to the direction of $-\xi_1$. At first glance, we may think that it is a simple task: we just have to copy the idea of K\H{o}nig's lemma twice, first for the geodesics to the points $\xi_1, 2\xi_1, ...$, and then to the points $-\xi_1, -2\xi_1, ...$, and thus we obtain two geodesic rays. However, we cannot guarantee that they are distinct, even though our instinct might say that they should be if there are not large deviations between the values in $A$. The reason behind this difficulty is that in order to control this property, seemingly we would need some knowledge about the passage times of infinitely many edges, which is something we cannot obtain in a large set. Thus this remains an open question.

\section{The behavior of $\frac{T(0, \mu x)}{\mu |x|}$}

Before proving Theorem 1.4, it is worth mentioning that a sequence of the form $\frac{T(0,\mu_k x)}{\mu_k |x|}$ cannot have a limit smaller than $\inf A$ or larger than $\sup A$, regardless of which configuration we observe. Indeed, choose $\mu$ large and let us denote by $p(\mu,x)$ the lattice point with the property $\mu x \in p(\mu,x)+[0,1)^d$, that is the lattice point which was used to define the passage time $T(0,\mu x)$. Then we have $\left|p(\mu,x)-\mu x\right|<d$. Thus we have
\begin{displaymath}
(\mu |x| -d)\inf A \leq T (0, \mu x) \leq (\mu |x| +d) \sup A,
\end{displaymath}
where we obtain the first inequality by considering any path from $0$ to $p(\mu,x)$ and the second one by considering a path between these points with minimal $\ell_1$-length. A simple rearrangement verifies our claim.

\begin{proof}[Proof of Theorem 1.4]

By a simple rescaling it is easy to see that it suffices to prove the statement for $x\in\mathbb{R}^d$ with $|x|=1$. Indeed, if for a given $\lambda$ the sequence of coefficients $\mu_k$ yields the given limit for the point $\frac{x}{|x|}$ then the sequence of coefficients $\mu_k|x|$ will be fine for the point $x$. In the spirit of this remark let us fix $x\in\mathbb{R}^d$ with $|x|=1$. We say that $\frac{T(0,\mu x)}{\mu}$ is a normalized passage time in the direction of $x$.

Let us denote by $S$ the set of configurations that are "bad" for us, namely the subset of $\Omega$ in which there exists some (finite or infinite) $\lambda$ with ${\inf A} \leq \lambda \leq {\sup A}$ such that there is no sequence $(\mu_k)_{k=1}^{\infty}$ with $\mu_k\to\infty$ such that $\lim_{k\to\infty}\frac{T(0,\mu_kx)}{\mu_k}=\lambda$. In this case, there is surely such a finite $\lambda$, thus in our further arguments we think about $S$ this way. Our aim is to express $S$ as a countable union of sets which are easier to handle and prove that these sets are nowhere dense. Having this purpose in mind, we will denote by $S(\lambda,\delta,M)$ the set of configurations in which for any $\mu>M$ we have that the distance of $\frac{T(0,\mu_kx)}{\mu_k}$ and $\lambda$ is larger than $\delta$. The following equation clearly holds:
\begin{displaymath}
S=\bigcup_{\inf A < \lambda < \sup A}\bigcup_{\delta>0}\bigcup_{M>0}S\left(\lambda,\delta,M \right).
\end{displaymath}
\noindent Indeed, by the definition of convergence if there is no appropriate sequence of coefficients for a given $\lambda\in(\inf A, \sup A)$ then there exists a neighborhood of it such that $\frac{T(0,\mu x)}{\mu}$ is not in this neighborhood for large enough $\mu$. However, by basic separability arguments on the real line we have that it further equals
\begin{displaymath}
S=\bigcup_{\lambda\in\mathbb{Q}, \inf A < \lambda < \sup A}\bigcup_{n\in\mathbb{N}}\bigcup_{m\in\mathbb{N}}S\left(\lambda,\frac{1}{n},m \right),
\end{displaymath}
\noindent which is a decomposition we pursued.

Having this knowledge it suffices to prove that all the sets $S\left(\lambda,\frac{1}{n},m \right)$ are nowhere dense. In order to prove this, fix $\lambda,n,m$, and fix real numbers $a,b$ with
\begin{displaymath}
\inf A \leq a < \lambda < b \leq \sup A.
\end{displaymath}
(This step has importance only if $A$ is unbounded, and its sole technical role is that we cannot calculate with $\sup A$ in this case, thus it needs to be replaced by a finite quantity.)

Clearly it suffices to prove our claim for large enough $n$, as for fixed $\lambda$ and $m$ the sequence $S\left(\lambda,\frac{1}{n},m \right)$ is growing as $n$ tends to infinity. Thus without loss of generality it suffices to consider the case when $a+\frac{1}{n}<\lambda<b-\frac{1}{n}$.

As usual, fix $U$ to be a cylinder set, and denote the set of edges belonging to nontrivial projections of $U$ by $E_U=\{e_1, e_2, ..., e_k\}$. As in the proof of Theorem 2.1, we can construct a smaller cylinder set by shrinking the projections $U_{e_1}, ..., U_{e_k}$, such that all of these projections are bounded in $\mathbb{R}$. Again, we denote these new projections by $U'_{e_i}$, $i=1,...,k$, and the cylinder set defined by them by $U'$. Then for any configuration in $U'$, the sum of passage times over the edges $e_1,...,e_k$ is bounded by a constant $C$. Our goal is to find a cylinder set $V\subseteq{U'}$ and some $\mu>m$ such that the distance of $\frac{T(0,\mu x)}{\mu}$ and $\lambda$ is at most $\frac{1}{n}$ for any configuration in $V$. We state that for suitably large $\mu$ it is possible to find such $V$. Consider a large $\mu>1$, its exact value is to be determined later. 

Now fix a path $\Gamma_0$ with minimal $\ell_1$-length from the origin to $p(\mu,x)$. Roughly we would like to define $V$ such that it has nontrivial projections to the edges in $E_U$ and to the edges in a large box $K$ containing $0$ and $p(\mu,x)$. (The size of $K$ is also to be fixed later.) Concerning the edges in $\Gamma_0\setminus E_U$, we would like to define the projections so that the passage time of $\Gamma_0$ is close to $\lambda\mu$, by having projections close to $a$ or $b$ with a suitable frequency. For the other edges in $K$ we would like to have projections close to $b$ in order to guarantee that the passage time between 0 and $p(\mu,x)$ is not reduced too much by another path.

Rigorously speaking, choose $\mu$ sufficiently large so that $|p(\mu,x)|=N_1+N_2$ for some positive integers satisfying
\begin{displaymath}
\frac{aN_1 + bN_2}{|p(\mu,x)|}\in\left(\lambda-\frac{1}{4n},\lambda+\frac{1}{4n}\right).
\end{displaymath}
\noindent As $|p(\mu,x)|$ can be arbitrarily large and the length of the interval we aim at is fixed, it is simple to see that we can choose $\mu,N_1,N_2$ to satisfy this relation. Moreover, as the distance of $|p(\mu,x)|$ and $\mu$ is bounded by $d$, for suitably large $\mu$ this yields
\begin{equation}
\frac{aN_1 + bN_2}{\mu}\in\left(\lambda-\frac{1}{2n},\lambda+\frac{1}{2n}\right).
\end{equation}
\noindent Now we choose $N_1$ edges of $\Gamma_0$, and for the ones not in $E_U$, we require $V$ to have projection $(a-\varepsilon_e,a +\varepsilon_e)\cap A$ to any such edge $e$, such that the sum of these $\varepsilon_e$s is at most $\frac{1}{4n}$. These are the cheap edges. We proceed similarly for all the other edges in $K$: for the ones not in $E_U$, we require $V$ to have projection $(b-\varepsilon_e,b +\varepsilon_e)\cap A$ to any such edge $e$, such that the sum of these $\varepsilon_e$s is at most $\frac{1}{4n}$. These are the expensive edges, and by the choice of $n$, they are bounded away from the cheap ones. As the number of edges in $E_U$ is fixed and $N_1,N_2$ can grow arbitrarily large for large $|p(\mu,x)|$, the projection to the majority of the edges in $\Gamma_0$ will be either cheap or expensive. We fix $K$ now: define it such that any path leaving $K$ contains at least $|p(\mu,x)|$ expensive edges.

Now our only remaining task is to estimate the passage time between 0 and $p(\mu,x)$ for configurations in $V$. Our aim is to verify that we have
\begin{equation}
T(0,p(\mu,x))\in\left[\mu\left(\lambda-\frac{1}{n}\right),\mu\left(\lambda+\frac{1}{n}\right)\right],
\end{equation}
\noindent which would follow from
\begin{equation}
\tau(\Gamma)>\mu\left(\lambda-\frac{1}{n}\right)
\end{equation}
\noindent for any path $\Gamma$ from $0$ to $p(\mu,x)$ and
\begin{equation}
\tau(\Gamma_0)<\mu\left(\lambda+\frac{1}{n}\right).
\end{equation}
\noindent In order to check (5), consider now any path $\Gamma$ from 0 to $p(\mu,x)$. If $\Gamma$ leaves $K$, it contains at least $N_1+N_2$ expensive edges, which results in
\begin{displaymath}
\frac{\tau(\Gamma)}{\mu}\geq\frac{b(N_1+N_2)-\frac{1}{4n}}{\mu}>\frac{aN_1 + bN_2-\frac{1}{4n}}{\mu}>\lambda-\frac{1}{n},
\end{displaymath} 
\noindent by (3), $\mu>1$ and the condition on the expensive edges. Thus we have (5) for these paths. Assume now that $\Gamma$ stays in $K$. Then $|\Gamma|\geq |p(\mu,x)|$, and at most $k$ edges of $\Gamma$ is in $E_U$. Thus $\Gamma$ has at least $N_1+N_2-k$ edges which are either cheap or expensive. As amongst these at most $N_1$ are cheap, we have the following lower bound on the passage time of $\Gamma$ if we forget about the edges in $E_U\cap\Gamma$ and consider the trivial lower estimates for the number and passage times of cheap and expensive edges:
\begin{displaymath}
\frac{\tau(\Gamma)}{\mu}\geq\frac{aN_1+b(N_2-k)-\frac{1}{4n}}{\mu}>\lambda-\frac{1}{n},
\end{displaymath}
\noindent by (3) for large enough $\mu$ as $\frac{bk}{\mu}$ tends to 0. It verifies (5) for any path from 0 to $p(\mu,x)$, hence it remains to show (6). However, it can be done similarly. We know that $\Gamma_0$ contains at most $N_1$ cheap edges, $N_2$ expensive edges, and the sum of passage times on the edges in $E_U\cap\Gamma$ is bounded by $C$ for any configuration in $V$. Thus we have
\begin{displaymath}
\frac{\tau(\Gamma)}{\mu}\leq\frac{aN_1+bN_2+\frac{1}{4n}+C}{\mu}<\lambda+\frac{1}{n},
\end{displaymath}
by (3) for large $\mu$, which verifies (6), and concludes the proof. \end{proof}

\section{The behavior of $\frac{B(t)}{t}$}

Before proving Theorem 1.5, we would like to explain its conditions. Requiring connected and closed limit sets is completely reasonable, as $B(t)$ is always connected, however, to require them to be bounded is not natural. The reason behind this is that the case of the unbounded closed sets seems to be much more difficult to handle: similar difficulties arise as in proving the existence of distinct geodesic rays. More precisely, if we want to copy our argument given for compact sets, at a point we cannot proceed as we would need some knowledge about infinitely many passage times which we lack on the complement of a nowhere dense set.

The necessity of the conditions about containing $D_\frac{1}{\sup A}$ and being contained by $D_\frac{1}{\inf A}$ can be verified similarly as the necessity of the conditions of Theorem 1.4. For example even if every passage time would be $\inf{A}$, which yields that $B(t)$ is as large as can be for each $t$, the limit of $\frac{B(t)}{t}$ would be $D_\frac{1}{\inf A}$, and not larger.

Now let us observe the definition of $\mathcal{P}_A^d$. If $\inf A=0$, we have $\overline{\mathcal{P}_A^d}=\mathcal{K}_A^d$, thus it does not require further explanation. However, we state that for any $A$ and $K\notin\mathcal{K}_A^d\setminus\overline{\mathcal{P}_A^d}$ there is no configuration for which there exists a sequence $(t_n)_{n=1}^\infty$ with the given properties. Assume the converse. Denote by $\tilde{B}(t)$ the subgraph of $\mathbb{Z}^d$ which is accessible from the origin in time $t$. Then as the Hausdorff distance of $B(t)$ and $\tilde{B}(t)$ is uniformly bounded by a constant dependent only on the dimension, we have that $\frac{\tilde{B}(t_n)}{t_n}$ also converges to $K$ in Hausdorff distance. However, we know that $\tilde{B}(t_n)$ is a connected subgraph of $\mathbb{Z}^d$, and each of its points is accessible from the origin using a path with $\ell_1$-length $\frac{t_n}{\inf A}$. Thus any point of $\frac{\tilde{B}(t_n)}{t_n}$ is accessible from the origin using a topological path, which stays in the set, and has $\ell_1$-length at most $\frac{1}{\inf A}$. It easily yields that 
\begin{displaymath}
\frac{\tilde{B}(t_n)}{t_n}\in\overline{\mathcal{P}_A^d},
\end{displaymath}
\noindent as we can shrink $\frac{\tilde{B}(t_n)}{t_n}$ a bit, we get a set in $\mathcal{P}_A^d$. As a consequence, $K\in\overline{\mathcal{P}_A^d}$ as the Hausdorff limit of the sets $\frac{\tilde{B}(t_n)}{t_n}$, a contradiction. This argument shows that we cannot have higher hopes than converging to sets in $\overline{\mathcal{P}_A^d}$. To conclude this remark, we point out that $\overline{\mathcal{P}_A^d}$ contains certain natural classes of sets, even if $\inf A\neq 0$. First of all, it is quite obvious that it contains all the convex sets of $\mathcal{K}_A^d$. Moreover, it contains the star domains of $\mathcal{K}_A^d$ with respect to the origin. We also mention a less natural class: we introduce the notion of star domains in $\ell_1$ sense. The set $S\subseteq\mathbb{R}^d$ is a star domain with respect to $x_0\in{S}$ in $\ell_1$ sense (or generalized star domain with respect to $x_0$), if for any $x\in{S}$ there is a topological path from $x_0$ to $x$ in $S$ with $\ell_1$-length $|x-x_0|$. In other words, each of the coordinate functions of the topological path are monotone. We denote the subset of $\mathcal{K}_A^d$ containing the generalized star domains with respect to $0$ by $\mathcal{K}_A^{d,*}$. Then $\mathcal{K}_A^{d,*}\subseteq \mathcal{P}_A^d$ also holds.

Finally, a few words should also be said about the condition $\inf A=0$ or $\sup A=\infty$. Sadly, we cannot say much about the case when $A$ is bounded away from both $0$ and $\infty$ if $d>1$. We would like to highlight though that the statement of Theorem 1.5 does not hold in this form by giving an example for $d=2$, which is easy to modify for higher dimensions. We formulate this claim as a proposition.

\begin{prop} For suitable $A$, there exists $K\in\overline{\mathcal{P}_A^{d}}$ such that there is no configuration in $\Omega$ and a sequence $(t_n)_{n=1}^{\infty}$ tending to infinity with $\frac{B(t_n)}{t_n} \to K$. \end{prop}

\begin{proof} Let $A=\{1,2\}$, and let $K=D_{\frac{1}{2}}\cup[0,\xi_1]$. Then $K\in\overline{\mathcal{P}_A^{d}}$ clearly holds, as $K$ is a star domain with respect to $0$. We state there is no configuration in $\Omega$ and a sequence $(t_n)_{n=1}^{\infty}$ tending to infinity with $\frac{B(t_n)}{t_n} \to K$. Assume the converse: there exists such a configuration and such a sequence of times. Then by the condition $\frac{B(t_n)}{t_n} \to K$, there exists a sequence of points $x_n=t_n\xi_1+o(t_n)v_n$, where $|v_n|=1$, and a path $\Gamma_n$ from 0 to $x_n$ with passage time $\tau\left(\Gamma_n\right)=t_n+o(t_n)$. (Here $o(t_n)$ denotes a sequence of quantities which satisfies $\frac{o(t_n)}{t_n}\to 0$ as $n\to\infty$.) This guarantees that $\Gamma_n$ contains at most $o(t_n)$ edges with passage time $2$. Moreover, for large enough $n$, these paths cross the boundary of $D_{\frac{1}{2}t_n}$ at some point $y_n$. Denote the piece of $\Gamma_n$ from $0$ to $y_n$ by $\Gamma_n'$. As $\Gamma_n'$ also contains at most $o(t_n)$ edges with passage time $2$, it is simple to check that it guarantees
\begin{displaymath}
T(0,y_n)\leq \tau\left(\Gamma_n'\right) = \frac{t_n}{2}+ o(t_n).
\end{displaymath}
\noindent Without loss of generality, we can assume that each $y_n$ lies in the upper half-plane. Based on the previous inequality, for arbitrary fixed $\alpha>0$ we have
\begin{equation}
\begin{split}
T\left(0,y_n+[\alpha t_n](\xi_1+\xi_2)\right) & \leq T(0,y_n)+T\left(y_n,y_n+[\alpha t_n](\xi_1+\xi_2)\right)\\ & \leq \left(\frac{1}{2}+4\alpha\right)t_n + o(t_n),
\end{split}
\end{equation}
\noindent if we estimate the second passage time of the middle expression by $4\alpha t_n$, which is a valid upper bound by the choice of $A$ and the $\ell_1$ distance of the two points observed. For large $n$ and small enough $\alpha$, it is strictly smaller than $t_n$, thus we have that $z_n=y_n+[\alpha t_n](\xi_1+\xi_2)$ is in $B(t_n)$. Moreover, for large enough $n$, the Euclidean distance of $z_n$ from both the first coordinate axis and $\partial D_{\frac{1}{2}t_n}$ is at least $\frac{\alpha t_n}{2}$. Thus the distance of $\frac{z_n}{t_n}$ from $K$ is at least $\frac{\alpha}{2}$ for large $n$. On the other hand, the sequence $\left(\frac{z_n}{t_n}\right)_{n=1}^{\infty}$ is in $D_1$, thus it has a convergent subsequence with limit $z\in{D_1}$ with distance at least $\frac{\alpha}{2}$ from $K$. However, $z$ is contained by the Hausdorff limit of $\frac{B_{t_n}}{t_n}$ by $z_n\in B_{t_n}$, which is $K$, a contradiction. \end{proof}

This proposition shows that in the cases not handled by Theorem 1.5, we need to modify the statement itself. Requiring convexity might be an attractive idea, as one might feel that in the example above the failure is somewhat caused by the lack of it, however, it is not complicated to construct configurations in which $\frac{B_t}{t}$ tends to a concave shape. Thus it might not be the proper way to overcome this difficulty. On the other hand, we have not even managed to prove the analogous statement for the convex sets of $\mathcal{K}_A^d$. Hence there is certainly a room for work on this question.

Now we turn our attention to the proof of Theorem 1.5. Denote by $\mathcal{P}_{A,0}^d$ the set that contains those sets of $\mathcal{P}_A^d$ which can be expressed as the closure of a connected open set. It is easy to see that $\mathcal{P}_{A,0}^d$ is dense in $\mathcal{P}_{A}^d$. Indeed, if $K\in\mathcal{P}_A^d$, denote by $K(r)$ the set of points which are in $D_\frac{1}{\inf A}$ and at most $r$ apart from $K$ in $\ell_1$. Then for sufficiently small $r$ the relation $K(r)\in\mathcal{P}_A^d$ holds: for any point $x\in{K(r)}$ we might choose $x'\in{K}$ within distance $r$. Then there is a topological path from the origin to $x'$ in $K$ of $\ell_1$-length at most $\frac{1}{\inf A} - \alpha_K$, which can be continued by a line segment of $\ell_1$-length $r$ to $x$. (Here we use the fact that $D_\frac{1}{\inf A}$ is convex.) Moreover, $K(r)$ is the closure of a connected open set: by the compactness of $K$, it is simply the closure of $K_0(r)$, the set of points which are in $D_\frac{1}{\inf A}$ and less than $r$ apart from $K$ in $\ell_1$.

As a consequence of the previous remark, if we could verify the modified statement of Theorem 1.5 which we obtain by replacing $\mathcal{P}_A^d$ by $\mathcal{P}_{A,0}^d$, that would be sufficient. We also recall that instead of the desired convergence of $\frac{B(t_n)}{t_n}$, it suffices to prove the same for $\frac{\tilde{B}(t_n)}{t_n}$.

Now by a standard argument about the separability of $\mathcal{P}_{A,0}^d$ it suffices to prove that for a given set $K\in\mathcal{P}_{A,0}^d$ we can find a suitable sequence of times in a residual set of $\Omega$. (We know that $\mathcal{P}_{A,0}^d$ is separable as it is a subspace of the separable $\mathcal{K}_A^d$, which is a metric space.) Our proof will rely on constructing cylinder sets in which we have a large control on $\tilde{B}(t)$. In other words, we desire to construct subgraphs of $\mathbb{Z}^d$ which are close to $tK$. To formalize this idea, we will need the following lemma:

\begin{lemma} Let $K\in\mathcal{P}_{A,0}^d$ and $\varepsilon>0$ fixed. Denote by $G_n$ the embedded graph whose vertices are the vertices of $\frac{\mathbb{Z}^d}{n}$ in $K$, and edges are those nearest neighbor edges which lie entirely in $K$. Then for infinitely many $n$ the graph $G_n$ has a connected subgraph $H_n$ such that it contains all the vertices and edges in $D_{\frac{1}{\sup A}}$, satisfies $d_H(K,H_n)<\varepsilon$, and to all of its vertices there is a path from the origin of $\ell_1$-length smaller than $\frac{1}{\inf A}$. \end{lemma}

If $\inf A=0$, we have $\mathcal{P}_{A,0}^d=\mathcal{K}_{A,0}^d$, and the last condition on $H_n$ is tautological. In this neater form, we find this lemma interesting in its own right as a nice exercise of a course in analysis. (It is likely to be known in some form, but we could not find a reference for it.)

\begin{proof}[Proof of Lemma 5.2]
As $\inte K$ is a connected open set, the points with rational coordinates in $\inte K$ form a dense subset of $K$. Let us consider now a open ball of radius $\varepsilon$ centered at each point with rational coordinates in $\inte K$. These balls give an open cover of the compact set $K$, thus we might choose a finite cover. Denote the centers of these balls by $v_1,...,v_m$. The coordinates of these points might have only finitely many distinct denominators. Thus if $n$ is chosen as a common multiple of them, $H_n$ can contain all these points, which guarantees $d_H(K,H_n)<\varepsilon$. What remains to show that is for large enough such $n$, it is possible to choose a connected $H_n$ satisfying the condition about the lengths of paths such that it contains the points $v_1,...,v_m$, and the vertices and edges in $D_{\frac{1}{\sup A}}$.

As $K\in\mathcal{P}_{A}^d$, there are topological paths $\gamma_1,...\gamma_m:[0,1]\to \inte K$ from $0$ to $v_1,...v_m$ with $\ell_1$-length less than $\frac{1}{\inf A}$. As each of the sets $\gamma_i\left([0,1]\right)$ are compact and contained by $\inte K$, it is possible to choose $r>0$ such that their neighborhoods of radius $r$ are also contained by $\inte K$. Moreover, by the definition of $\ell_1$-length we might choose points on $\gamma_i\left([0,1]\right)$ such that they can be connected by a broken line $L_i$ with pieces parallel to the coordinate vectors, and its length is also less than $\frac{1}{\inf A}$. Furthermore, by the existence of $r$, if we choose a suitably fine partition of $\gamma_i\left([0,1]\right)$, we might have $L_i\in \inte K$. For the sake of simplicity, denote the vertices of $L_i$ by $0=p_1,...,p_k=v_i$.  Now for $\beta>0$ fixed, we can choose $n$ so large that $\frac{\mathbb{Z}^d}{n}$ has vertices closer than $\beta$ in $\ell_1$ to any vertex of $L_i$. Denote such vertices of $\frac{\mathbb{Z}^d}{n}$ by $0=q_1,...,q_k=v_i$. Also if $n$ is large enough, if we consider the smallest lattice hypercubes of $\frac{\mathbb{Z}^d}{n}$ crossed by $L_i$, they are still in $\inte K$, and $q_1,...,q_k$ might be chosen to be the vertices of these cubes. Thus using the edges of these cubes we can find a path $\Gamma_{i,n}$ in $G_n$ from $0$ to $v_i$, which stays in $\inte K$, and optimal in $\ell_1$ between any vertices $q_j$ and $q_{j+1}$. Hence we can deduce by triangle inequality that
\begin{displaymath}
|\Gamma_{i,n}|=\sum_{j=1}^{k-1}|q_j q_{j+1}|\leq\sum_{j=1}^{k-1}|p_j p_{j+1}|+2\sum_{j=1}^{k}|p_j q_{j}|\leq|L_i|+2k\beta.
\end{displaymath}
As $k$ is fixed and $\beta$ can be arbitrarily small, it guarantees that for large enough $n$ the length $|\Gamma_{i,n}|$ is less than $\inf A$. We can define $H_n$ for infinitely many $n$ appropriately based on this argument: we require it to contain all the vertices and edges of $\frac{\mathbb{Z}^d}{n}$ in $D_{\frac{1}{\sup A}}$, it is clearly connected and all the vertices are accessible by a path of $\ell_1$-length less than $\frac{1}{\inf A}$. Furthermore, we require it to contain the vertices $v_1,...,v_m$ and the paths $\Gamma_{i,n}$, which does not mess up the condition about the distance of vertices from the origin, and guarantees the bound on the Hausdorff distance. \end{proof}

\begin{proof}[Proof of Theorem 1.5]

By our previous remarks, it suffices to prove that if $K\in\mathcal{P}_{A,0}^d$, then in a residual set of $\Omega$ there exists a suitable sequence $t_n\to \infty$ with $\frac{\tilde{B}(t_n)}{t_n}\to K$. Denote the set of configurations not having this property by $F(K)$. Then by the definition of convergence, $F(K)$ can be expressed as a countable union as follows:
\begin{displaymath}
F(K)=\bigcup_{i=1}^{\infty}\bigcup_{m=1}^{\infty}F\left(K,\frac{1}{i},m \right),
\end{displaymath}
where $F(K,\varepsilon,\mu_0)$ stands for the set of configurations in which for any $\mu>\mu_0$ we have
\begin{displaymath}
d_H\left(K,\frac{\tilde{B}(\mu)}{\mu}\right)>\varepsilon.
\end{displaymath}
Verifying that $F\left(K,\frac{1}{i},m \right)$ is nowhere dense for each $i,m$ would conclude the proof. Clearly it suffices to do so for large enough $i,m$.

As usual, fix $U$ to be a cylinder set, and denote the set of edges belonging to nontrivial projections of $U$ by $E_U=\{e_1, e_2, ..., e_k\}$. As in the proof of Theorem 2.1, we can construct a smaller cylinder set by shrinking the projections $U_{e_1}, ..., U_{e_k}$, such that all of these projections are bounded in $\mathbb{R}$. Again, we denote these new projections by $U'_{e_i}$, $i=1,...,k$, and the cylinder set defined by them by $U'$. Then for any configuration in $U'$, the sum of passage times over the edges $e_1,...,e_k$ is bounded by a constant $C$. Our goal is to find a cylinder set $V\subseteq{U'}$ and some $\mu>m$ such that the Hausdorff distance of $\frac{\tilde{B}(\mu)}{\mu}$ and $K$ is at most $\frac{1}{i}$ for any configuration in $V$. We distinguish the cases based on the value of $\inf A$ and $\sup A$. The idea will be the same in the three cases, but the realization will vary.

\begin{enumerate}[(i)]

\item Assume first that $\inf A=0$ and $\sup A=\infty$, as technically it is the easiest. We pursue $\mu$ as a large enough $n\in\mathbb{N}$ for which $n>C$ and which satisfies Lemma 5.1 with $\varepsilon=\frac{1}{2i}$. Now we try to choose $V$ such that for any configuration in $V$, the set $\tilde{B}(n)$ is close to $nH_n$, which is a subgraph of $\mathbb{Z}^d$. For this aim, denote the edge set of $nH_n$ by $E(nH_n)$. For any edge $e\in E(nH_n)\setminus E_U$ we define $V_e$ to be $[0,\varepsilon_e)\cap A$, where the $\varepsilon_e$s are small enough to have a smaller sum than $n-C$. Furthermore, for any further edge $e$ leaving the graph $nH_n$ or neighboring to one of the edges in $E_U$, we define $V_e$ to have strictly larger elements than $n$. By the first part of the definition $nH_n\subseteq\tilde{B}(n)$ obviously holds. Furthermore, $\tilde{B}(n)$ may differ from $nH_n$ in only the edges of $E_U$, which yields that their Hausdorff distance is at most $k$. As a consequence, since the Hausdorff distance of $nH_n$ and $K$ is at most $\frac{n}{2i}$, by triangle inequality we have that 
\begin{displaymath}
d_H\left(K,\frac{\tilde{B}(n)}{n}\right)\leq\frac{1}{2i}+\frac{k}{n}\leq\frac{1}{i},
\end{displaymath}
if $n$ is large enough. It concludes the proof in this case.

\item If $\inf A=0$ and $A$ is bounded, the proof relies on the same concept, but our task is a bit more difficult. We choose $n$ as in (i), and look for $V$ with a similar property. If $n$ is large enough, we have $E_U\subseteq nH_n$. Choose $N\in\mathbb{N}$ with $N\sup A > C$. For an edge $e$ in $D_{\frac{n}{\sup A} -N}$, which is not contained by $E_U$, we define $V_e$ to be $(\sup A -\varepsilon_e, \sup A)\cap A$, where the $\varepsilon_e$s are small enough, they are to be fixed later. Thus these are expensive edges. For any other edge $e$ of $E(nH_n)\setminus E_U$ we stick to the definition in (i): $V_e=[0,\varepsilon_e)\cap A$, here $\varepsilon_e$ is small again, these are cheap edges. Finally, for any further edge $e$ with distance at most $2N$ from the graph $nH_n$, we define $V_e$ to be $(\sup A -\varepsilon_e, \sup A)\cap A$, hence these are expensive edges again. Now if we consider any point $x\in nH_n$, there is a path $\Gamma$ to it from the origin which might use the edges of $E_U$, and uses at most $\left[\frac{n}{\sup A} -N\right]$ expensive edges. All the other edges in $\Gamma$ are cheap. Thus by the definition of $N$, if we choose the $\varepsilon_e$s to have small enough sum, the passage time of $\Gamma$ is bounded by $n$ for any configuration in $V$, which results in $x\in\tilde{B}(n)$. Furthermore, if a point $x$ is further from $nH_n$ than $2N$, any path from $0$ to $x$ uses more than $\frac{n}{\sup A}$ expensive edges, which yields that if the $\varepsilon_e$s have small enough sum, $x\notin\tilde{B}(n)$. As a consequence, the Hausdorff distance of $\tilde{B}(n)$ and $nH_n$ is at most $2N$, which is fixed. The final step is the same triangle inequality as in (i).

\item Assume $\inf A>0$ and $\sup A=\infty$. We might attempt to copy the argument of (i). The only difficulty is that we have to replace the projections $[0,\varepsilon_e)\cap A$ by $(\inf A, \inf A +\varepsilon_e)\cap A$. Now by the last condition on $H_n$ in Lemma 5.1, for any of vertex $x$ of $nH_n$ there exists a path $\Gamma$ from $0$ to $x$ with $\ell_1$-length less than $\frac{n}{\inf A}$. Thus if $C<N \inf A$, we might obtain that for any such $x$ with $|x|<\frac{n}{\inf A} - N$, the passage time of $\Gamma$ is at most $n$ for a good choice of $\varepsilon_e$. This means that $\tilde{B}(n)$ contains all the points of $nH_n$ for any configuration in $V$, except for possibly those ones which are closer to $\partial D_\frac{n}{\inf A}$ than $N$. On the other hand, $\tilde{B}(n)$ cannot contain a point which is further from $nH_n$ than $k$. Thus the Hausdorff distance of n$H_n$ and $\tilde{B}(n)$ can be bounded by $N+k$, and the proof might be finished using the triangle inequality. \end{enumerate} \end{proof}

\section{Concluding remarks}

The goal of the present paper was not to widen our knowledge about the vast field of the probabilistic setup, but to introduce another nice topic. In certain questions handled in this work there is room for improvement: the constant used in Theorem 1.2 or the bound on the number of distinct geodesic rays in Theorem 1.3 might be lowered by a smarter geometric argument. Moreover, it would be nice to know the kind of configurations in which there are multiple geodesic rays, and Theorem 1.5 should also be generalized to any set $A$. Other problems may also be borrowed from the original first passage percolation. Concerning certain questions, one may observe other infinite graphs instead of the lattice. Our wildest hope is that some of the appearing ideas might be recycled somehow in probability theory, even though we find it unlikely due to the quite different nature of the areas.

\subsection*{Acknowledgements} 

I am highly grateful to Zolt\'an Buczolich for the time he spent with proofreading this paper and spotting a flaw in one of the proofs. Moreover, I am thankful to L\'aszl\'o Erd\H{o}s, Jan Maas, and Peter Nejjar for introducing me to the beautiful theory of first passage percolation.

\end{document}